\documentclass[11pt,english]{article}
\def\margin_comment#1{\marginpar{\sffamily{\tiny #1\par}\normalfont}}
\usepackage{graphicx}
\usepackage{setspace}
\date{}
\usepackage{amsmath,amssymb,amsthm}
\usepackage{babel}
 \newtheorem{thm}{Theorem}[section]
 \numberwithin{equation}{section} 
 \numberwithin{figure}{section} 
 \theoremstyle{plain}
  \newtheorem*{thm*}{Theorem}
 \theoremstyle{definition}
\theoremstyle{plain}
\newtheorem{thm_A}{Theorem}
 \newtheorem*{defn*}{Definition}
 \theoremstyle{plain}
 \newtheorem{prop}[thm]{Proposition} 
 \theoremstyle{remark}
 \newtheorem{ex}[thm]{Example}
 \theoremstyle{remark}
 \newtheorem{rem}[thm]{Remark}
\theoremstyle{plain}
 \newtheorem{claim}[thm]{Claim}
\theoremstyle{plain}
 
 \theoremstyle{plain}
 \newtheorem{lem}[thm]{Lemma} 
 \theoremstyle{definition}
 \newtheorem{defn}[thm]{Definition}
\newtheorem*{acknowledgment}{Acknowledgment}

\AtBeginDocument{
  
}
\begin{document}
\title{Garside groups and Yang-Baxter equation}
\author{Fabienne Chouraqui}
\maketitle
\begin{abstract}
 We establish a one-to-one correspondence between  a class of Garside groups admitting a certain presentation and the structure groups of  non-degenerate, involutive  and braided set-theoretical solutions of the quantum Yang-Baxter equation.
We also characterize indecomposable solutions in terms of  $\Delta$-pure Garside  groups.
\end{abstract}

\section{Introduction}
The quantum Yang-Baxter equation is an equation in the field of mathematical physics and it lies in the foundation of the theory of quantum groups.
Let $R:V  \otimes V \rightarrow V  \otimes V$ be a linear operator, where $V$ is a vector space. The quantum Yang-Baxter equation is the equality $R^{12}R^{13}R^{23}=R^{23}R^{13}R^{12}$ of linear transformations on $V  \otimes V \otimes V$, where $R^{ij}$ means $R$ acting on the $i$th and $j$th components.\\
A set-theoretical solution of this equation  is a solution for which $V$ is a vector space spanned by a set $X$ and $R$ is the linear operator induced by a mapping $X \times X \rightarrow X \times X$. The study of these was suggested by Drinfeld \cite{drinf}. Etingof, Soloviev and Schedler  study set-theoretical solutions $(X,S)$ of the quantum Yang-Baxter equation that are  non-degenerate, involutive and braided \cite{etingof}. To each such solution, they associate a group called the structure group and they show that this group satisfies some properties.
They also give a classification of such solutions $(X,S)$ up to isomorphism,  when the cardinality of $X$ is  up to eight. As an example, there are 23 solutions for $X$ with four elements, 595  solutions for $X$ with six elements and  34528 solutions for $X$ with eight elements. In this paper, we establish  a one-to-one correspondence between non-degenerate, involutive and braided set-theoretical solutions of the quantum Yang-Baxter equation (up to isomorphism) and Garside presentations which satisfy some additional conditions up to t-isomorphism (a notion that will be defined below). The main result is as follows.
\begin{thm_A}
$(i)$ Let $X$ be a finite set, and $(X,S)$ be  a non-degenerate, involutive and braided set-theoretical solution of the quantum Yang-Baxter equation. Then the structure group of $(X,S)$ is a Garside group.\\ $(ii)$ Conversely, assume that  $\operatorname{Mon} \langle X\mid R \rangle$ is a  Garside monoid such that:\\
 - the cardinality of $R$ is $n(n-1)/2$, where $n$ is the cardinality of $X$ and each side of a relation in $R$ has length 2 and \\
- if the  word $x_{i}x_{j}$ appears in $R$, then it appears only once.\\
 Then there exists a function $S: X \times X \rightarrow X \times X$ such that $(X,S)$ is  a non-degenerate, involutive and braided set-theoretical solution and $\operatorname{Gp} \langle X\mid R \rangle$ is its structure group.
\end{thm_A}
The main idea of the proof is to express the right and left complement on the generators in terms of the functions that define $(X,S)$. Moreover, we prove that the structure group of a set-theoretical solution satisfies some specific constraints. Picantin defines the notion of $\Delta$-pure Garside group in \cite{picantin} and  he shows that the center of a $\Delta$-pure Garside group is a cyclic subgroup that is generated by some exponent of its Garside element.

\begin{thm_A}
Let $X$ be a finite set, and  $(X,S)$ be  a non-degenerate, involutive and braided set-theoretical solution of the quantum Yang-Baxter equation. Let $G$ be the structure group of $(X,S)$ and $M$ the monoid with the same presentation. Then \\
$(i)$ The right least common multiple of the elements in $X$ is a Garside element in $M$.\\
$(ii)$ The (co)homological dimension of  $G$ is equal to the cardinality of $X$.\\
$(iii)$ The  group $G$ is   $\Delta$-pure Garside  if and only if $(X,S)$ is indecomposable.
\end{thm_A}

Point \emph{$(i)$} above means that  $G$ is Garside in the restricted sense of \cite{deh_Paris}.
Let us observe  that, independently, Gateva-Ivanova and Van den Bergh define  in  \cite{gateva+van} monoids and groups of left and right I-type and they show that they yield solutions to the quantum Yang-Baxter equation. They show also that a monoid of left I-type is cancellative and has a group of fractions that is torsion-free and Abelian-by-finite. Jespers and Okninski extend their results in \cite{jespers+okninski}, and establish a correspondence between groups of I-type and the structure group of a non-degenerate, involutive and braided set-theoretical solution. Using our result, this makes a correspondence between groups of I-type and the class of Garside groups  studied in this paper.
They also remark that the defining presentation of a monoid of I-type satisfies the right cube condition, as defined by Dehornoy in \cite[Prop.4.4]{deh_complte}. So, the necessity of being Garside can be derived from the combination of the results from \cite{jespers+okninski,gateva+van}. Our methods in this paper are different as we use the tools of reversing and complement developed in the theory of Garside monoids and groups and our techniques of proof are uniform throughout the paper. It can be observed that our results   imply some earlier results by Gateva-Ivanova. Indeed, she  shows in \cite{gateva} that the monoid corresponding to a special case of  non-degenerate, involutive and braided set-theoretical solution (square-free) has a structure of distributive lattice with respect to left and right divisibility and that the left least common multiple of the generators is equal to their right least common multiple and she calls this element the principal monomial.\\
 The paper is organized as follows.
In section $2$, we give some preliminaries on Garside monoids.
In section $3$, we give the definition of the structure group of a non-degenerate, involutive and braided set-theoretical solution and we show that it is Garside, using the criteria developed by Dehornoy in \cite{deh_francais}.
This implies that this group is torsion-free from \cite{deh_torsion} and biautomatic from \cite{deh_francais}.
In section $4$, we show that  the right least common multiple of the generators is a Garside element and that the (co)homological dimension of the structure group of a non-degenerate, involutive and braided set-theoretical solution is equal to the cardinality of $X$. In section $5$,  we give the definition of a $\Delta$-pure Garside group and we show that the structure group of $(X,S)$ is   $\Delta$-pure Garside  if and only if $(X,S)$ is indecomposable.
In section $6$,  we establish a converse to the results of section $3$, namely that  a Garside  monoid satisfying some additional conditions defines a non-degenerate, involutive and braided set-theoretical solution of the quantum Yang-Baxter equation.
Finally, in section $7$, we address the case of non-involutive solutions. There, we consider the special case of  permutation solutions that are not involutive and we show that their structure group is Garside. We could not extend this result to general solutions, although we conjecture this should be true. At the end of the section, we give the form of a Garside element in the case of permutation solutions.\\

\begin{acknowledgment}
This work is a part of  the author's PhD research, done  at the Technion under
the supervision of Professor Arye Juhasz.  I am very grateful to
Professor Arye Juhasz, for his patience, his encouragement and his many helpful remarks.
I am also grateful to Professor Patrick Dehornoy for his comments on this paper.
\end{acknowledgment}

\section{ Garside monoids and groups}
All the definitions and results in this section are from
\cite{deh_francais} and  \cite{deh_livre}. In this paper, if the element $x$ of $M$ is in the equivalence class of the word $w$, we say that \emph{$w$ represents $x$}.
\subsection{Garside monoids}
Let $M$ be a monoid and let $x,y,z$ be elements in $M$. The element $x$ is \emph{a left divisor} of $z$ if there is an element $t$  such that $z=xt$ in $M$ and  $z$ is
\emph{a right least common multiple  (right lcm)} of $x$ and $y$ if $x$ and $y$ are left divisors of $z$ and
additionally  if there is an element $w$ such that $x$ and $y$ are left divisors of $w$, then
   $z$ is  left divisor of $w$. We denote it by $z=x\vee y$.
  \emph{The complement at right of
   $y$ on $x$} is defined to be an element $c\in M$ such that $x\vee y= xc$, whenever $x\vee y$ exists. We denote it by
   $c=x \setminus y$ and by definition, $x\vee y=x(x \setminus y)$.
   Dehornoy shows that if $M$ is left cancellative and $1$ is the unique invertible element in $M$, then the right lcm and the right complement  of two elements are unique, whenever they exist \cite{deh_francais}. We refer the reader to \cite{deh_livre,deh_francais} for the definitions of the left lcm and the left and right gcd of two elements.  An element $x$ in  $M$ is \emph{an atom}  if $x \neq 1$ and  $x=yz$ implies $y=1$ or $z=1$. The \emph{norm} $\parallel x \parallel$ of $x$ is defined to be the supremum of the lengths of the
decompositions of $x$ as a product of atoms. The  monoid  $M$ is
\emph{atomic} if $M$ is generated by its atoms and for every $x$
in $M$ the norm of $x$ is finite. It holds that if all the relations in $M$
are length preserving, then $M$ is atomic, since each element $x$ of $M$ has a finite norm as all the words which represent $x$ have the same length.

A monoid  $M$ is \emph{Gaussian} if $M$ is
atomic, left and right cancellative, and if any two elements in
$M$ have a left and right gcd and lcm. If $\Delta$ is an element in $M$, then $\Delta$ is a \emph{Garside element} if the left divisors of $\Delta$ are the same as the right divisors, there is a finite
number of them and they generate $M$. A monoid $M$ is \emph{Garside} if $M$ is Gaussian and it  contains a Garside element.  A group $G$ is  a \emph{Gaussian group} (respectively
a \emph{Garside group}) if there exists a Gaussian monoid $M$
(respectively a Garside monoid) such that $G$ is the fraction
group of $M$. A Gaussian monoid satisfies both left and right Ore's conditions, so it embeds in its group of fractions (see \cite{Clifford}).
As an example, braid groups  and Artin groups of finite type \cite{garside}, torus knot groups \cite{picantin_torus} are Garside groups.

\begin{defn}\cite[Defn.1.6]{deh_francais}\label{def_conditions}
Let $M$ be a monoid. $M$ satisfies:\\
 - $(C_{0})$ if $1$ is the unique invertible element in $M$.\\
- $(C_{1})$ if $M$ is left cancellative.\\
- $(\tilde{C_{1}})$ if $M$ is right cancellative.\\
- $(C_{2})$ if any two elements in $M$ with a right common multiple
admit
a right lcm.\\
- $(C_{3})$ if  $M$ has a finite generating set $P$ closed under
$\setminus$, i.e if $x,y \in P$ then $x \setminus y \in P$.
\end{defn}
\begin{thm}\cite[Prop. 2.1]{deh_francais} \label{gars_critere}
A monoid $M$ is a Garside monoid if and only if $M$ satisfies the
conditions $(C_{0})$, $(C_{1})$, $(\tilde{C_{1}})$, $(C_{2})$,
and  $(C_{3})$.
\end{thm}
\subsection{Recognizing Garside monoids}

Let $X$ be an alphabet and denote by $\epsilon$ the empty word in $X^{*}$. Let $f$ be  a partial function of $X\times X$ into $X^{*}$, $f$ is
a \emph{complement} on $X$ if  $f(x,x)=\epsilon$ holds for every
$x$ in $X$, and $f(y,x)$ exists whenever $f(x,y)$ does.
The congruence on  $X^{*}$ generated by the pairs $(xf(x,y),yf(y,x))$ with $(x,y)$ in the domain of $f$ is denoted by  ``$\equiv^{+}$''. The monoid  \emph{associated with $f$} is  $X^{*}/ \equiv^{+}$ or in other words the monoid $\operatorname{Mon}\langle    X\mid xf(x,y)=yf(y,x)\rangle$ (with $(x,y)$ in the domain of $f$). The complement mapping considered so far is defined on letters
only. Its extension on words is called \emph{word reversing} (see \cite{deh_livre}).

\begin{ex}
\label{example_struct_gp} Let $M$ be the monoid generated by $X=\{x_{1},x_{2},x_{3},x_{4},x_{5}\}$ and defined by the following $10$  relations.\\
$\begin{array}{ccccc}
  x^{2}_{1}=x^{2}_{2} &  x_{2}x_{5}=x_{5}x_{2}  &x_{1}x_{2}=x_{3}x_{4} &  x_{1}x_{5}=x_{5}x_{1}&
x_{1}x_{3}=x_{4}x_{2}\\ x^{2}_{3}=x^{2}_{4} &
x_{2}x_{4}=x_{3}x_{1} & x_{3}x_{5}=x_{5}x_{3}&
x_{2}x_{1}=x_{4}x_{3}  &x_{4}x_{5}=x_{5}x_{4}
\end{array}$\\
 Then the complement $f$ is defined totally on $X \times X$ and the monoid associated to $f$,
 $X^{*} /\equiv^{+}$, is $M$.
As an example,   $f(x_{1},x_{2})=x_{1}$  and
$f(x_{2},x_{1})=x_{2}$ are obtained from the relation
$x^{2}_{1}=x^{2}_{2}$,  since it holds that $f(x_{1},x_{2})=x_{1}\setminus x_{2}$.

\end{ex}

Let $f$ be a complement on $X$. For $u,v,w \in X^{*}$, $f$ is \emph{coherent} at $(u,v,w)$ if either
$( (u
\setminus v ) \setminus (u \setminus w))
\setminus((v \setminus u) \setminus(v \setminus
w))\equiv^{+} \epsilon$ \ holds, or neither of the words $( (u
\setminus v )\setminus (u\setminus w)) ,    ((v
\setminus u)\setminus(v \setminus w))$ exists. The complement
 \emph{$f$ is coherent} if it is coherent at every triple $(u,v,w)$  with $u,v,w \in X^{*}$.
 Dehornoy shows that if the monoid is atomic then it is enough to show  the coherence of $f$ on its set of atoms. Moreover, he  shows that if  $M$ is a monoid associated with a coherent complement, then
$M$ satisfies $C_{1}$  and $C_{2}$ (see \cite[p.55]{deh_livre}).

\begin{prop} \cite[Prop.6.1]{deh_francais}\label{atomic_coh}
Let $M$ be a monoid associated with a complement $f$ and assume
that $M$ is atomic. Then $f$ is coherent if and only if $f$ is
coherent on $X$.
\end{prop}
\begin{ex} In example \ref{example_struct_gp}, we  check if
$( (x_{1} \setminus
x_{2} )\setminus (x_{1}\setminus x_{3}))
\setminus((x_{2} \setminus x_{1})\setminus(x_{2}
\setminus x_{3}))= \epsilon$ holds in $M$. We have $x_{1}\setminus x_{2}=x_{1}$
and $x_{1}\setminus x_{3}=x_{2}$, so  $(x_{1}\setminus x_{2})\setminus(x_{1}\setminus x_{3})=x_{1}\setminus x_{2}=x_{1}$. Additionally, $x_{2}\setminus x_{1}=x_{2}$ and  $x_{2}\setminus x_{3}=x_{4}$, so $(x_{2}\setminus x_{1})\setminus(x_{2}\setminus x_{3})=
x_{2}\setminus x_{4}=x_{1}$. At last,
$((x_{1}\setminus x_{2})\setminus(x_{1}\setminus x_{3}))
\setminus
((x_{2}\setminus x_{1})\setminus(x_{2}\setminus x_{3}))=
x_{1}\setminus x_{1}=\epsilon $.
\end{ex}

\section{Structure groups are Garside}
\subsection{The structure group of a set-theoretical solution}
All the definitions and results in this subsection are from \cite{etingof}.

A set-theoretical solution  of the quantum Yang-Baxter equation is a pair $(X,S)$, where $X$ is a non-empty set and $S:X^{2}\rightarrow  X^{2}$ is  a bijection. Let  $S_{1}$ and $S_{2}$ denote the components of $S$, that is $S(x,y)=(S_{1}(x,y),S_{2}(x,y))$.
A pair $(X,S)$ is  \emph{nondegenerate} if the maps $X \rightarrow X$ defined by $x \mapsto S_{2}(x,y)$ and $x \mapsto S_{1}(z,x)$ are bijections for any fixed $y,z \in X$.  A pair $(X,S)$ is  \emph{braided} if $S$ satisfies the braid relation $S^{12}S^{23}S^{12}=S^{23}S^{12}S^{23}$, where the map $S^{ii+1}: X^{n} \rightarrow X^{n}$ is defined by $S^{ii+1}=id_{X^{i-1}} \times S\times id_{X^{n-i-1}} $, $i<n $.
A pair $(X,S)$ is \emph{involutive} if $S^{2}=id_{X^{2}}$, that is $S^{2}(x,y)=(x,y)$ for all $x,y \in X$.\\
Let $\alpha:X \times X \rightarrow X\times X$ be the permutation map, that is  $\alpha(x,y)=(y,x)$, and let $R=\alpha \circ S$. The map $R$ is called the \emph{$R-$matrix corresponding to $S$}.
 Etingof, Soloviev and Schedler show in \cite{etingof}, that $(X,S)$ is a braided set if and only
 if $R$ satisfies the quantum Yang-Baxter equation $R^{12}R^{13}R^{23}=R^{23}R^{13}R^{12}$, where
$R^{ij}$ means acting on the $i$th and $j$th components and that $(X,S)$ is a symmetric  set if and only if in addition $R$ satisfies the unitary condition $R^{21}R=1$. They define the  \emph{structure group $G$ of $(X,S)$} to be the group generated by the elements of $X$ and  with defining relations $xy=tz$ when $S(x,y)=(t,z)$. They show that if  $(X,S)$ is non-degenerate and braided then the assignment $x \rightarrow f_{x}$ is a right action of $G$ on $X$.\\

 We use the notation of \cite{etingof}, that is if  $X$ is  a finite set, then  $S$ is  defined by $S(x,y)=(g_{x}(y),f_{y}(x))$, $x,y$ in $X$.
 Here, if $X=\{x_{1},...,x_{n}\}$ is a finite set and $y=x_{i}$ for some $1\leq i\leq n$, then
 we write  $f_{i},g_{i}$ instead of $f_{y},g_{y}$  and $S(i,j)=(g_{i}(j),f_{j}(i))$. The following claim from \cite{etingof} translates the properties of a solution $(X,S)$ in terms of the functions $f_{i},g_{i}$ and it will be very useful in this paper.
\begin{claim}\label{debut_form}
(i) $S$ is non-degenerate $\Leftrightarrow$ $f_{i}$, $g_{i}$ are bijective, $1 \leq i \leq n$.\\
$(ii)$ $S$ is involutive $\Leftrightarrow$ $g_{g_{i}(j)}f_{j}(i)=i$  and $f_{f_{j}(i)}g_{i}(j)=j$,
$1 \leq i,j \leq n$.\\
$(iii)$ $S$ is braided  $\Leftrightarrow$ $g_{i}g_{j}=g_{g_{i}(j)}g_{f_{j}(i)}$, $f_{j}f_{i}=f_{f_{j}(i)}f_{g_{i}(j)}$,\\ and $f_{g_{f_{j}(i)}(k)}g_{i}(j)= g_{f_{g_{j}(k)}(i)}f_{k}(j)$, $1 \leq i,j,k \leq n$.
\end{claim}
\begin{ex} Let$X=\{x_{1},x_{2},x_{3},x_{4},x_{5}\}$ and
$S(i,j)=(g_{i}(j),f_{j}(i))$. Assume \\
 $\begin{array}{ccc}
f_{1}=g_{1}=(1,2,3,4)(5) & f_{2}=g_{2}=(1,4,3,2)(5)\\
f_{3}=g_{3}=(1,2,3,4)(5) & f_{4}=g_{4}=(1,4,3,2)(5)
\end{array}$\\
Assume also that the functions $f_{5}$ and $g_{5}$ are the identity on $X$.
Then a case by case analysis  shows that  $(X,S)$ is a non-degenerate, involutive and braided solution. Its structure group is generated by $X$ and defined by the $10$  relations described in example \ref{example_struct_gp}.
\end{ex}
\subsection{Structure groups are Garside}
  In this subsection, we prove the following result.
\begin{thm}\label{theo:garside}
The structure group $G$ of a non-degenerate, braided and involutive
set-theoretical solution of the quantum Yang-Baxter equation is a Garside group.
\end{thm}
In order to prove that the group $G$ is a Garside group, we
 show that \emph{the monoid $M$ with the same presentation} is
a Garside monoid. For that, we  use the Garsidity criterion
given in  Theorem \ref{gars_critere}, that is we show that
$M$ satisfies the conditions $(C_{0})$ , $(C_{1})$ , $(C_{2})$,
$(C_{3})$, and  $(\tilde{C_{1}})$. We refer the reader to \cite[Lemma 4.1]{gateva+van} for the proof that $M$ is right cancellative ($M$  satisfies $(\tilde{C_{1}})$).
We first show that $M$ satisfies the conditions $(C_{0})$. In order to do that,  we describe the defining relations in $M$ and as they are length-preserving, this implies that $M$ is atomic.
\begin{claim}\label{cl_compl}
Assume $(X,S)$ is non-degenerate.
Let $x_{i}$ and $x_{j}$ be different elements in $X$.  Then
there is exactly one defining relation $x_{i} a= x_{j} b$, where
$a,b$ are in $X$.
If in addition, $(X,S)$ is involutive then $a$ and $b$ are different.
\end{claim}
For a proof of this result, see  \cite[Thm. 1.1]{gateva+van}.
Using the same arguments, if $(X,S)$ is non-degenerate and involutive there are no relations of the form $ax_{i}=ax_{j}$, where $i \neq j$.
We have the following direct result from claim \ref{cl_compl}.
\begin{prop} \label{complement_f_defined}
Assume $(X,S)$ is non-degenerate and involutive. Then the complement $f$ is totally defined on $X \times X$, its range is $X$ and the monoid associated to $f$ is $M$. Moreover,  $M$ is atomic.
\end{prop}
Now, we show that $M$ satisfies the conditions $(C_{1})$,  $(C_{2})$  and $(C_{3})$.
From Proposition \ref{complement_f_defined}, we have that there is a one-to-one correspondence  between the complement $f$ and the monoid $M$ with the same presentation as the structure group, so we  say that $M$ is coherent (by abuse of notation). In order to show that the monoid $M$  satisfies the conditions
$(C_{1})$ and $(C_{2})$, we  show that $M$ is coherent  (see \cite[p.55]{deh_livre}).
Since $M$ is atomic, it is enough to check its coherence on $X$
(from Proposition \ref{atomic_coh}).
We show that  any triple of generators $(a,b,c)$   satisfies the following equation:
 $ (a \setminus b )\setminus (a\setminus c)=
(b \setminus a)\setminus(b \setminus c)$,  where  the equality is in the free monoid $X^{*}$,
 since  the range of $f$  is $X$. In  the following lemma, we establish a correspondence between the right complement of generators and the functions $g_{i}$ that define $(X,S)$.
\begin{lem}\label{form_compl}
Assume $(X,S)$ is non-degenerate. Let $x_{i}, x_{j}$ be different elements in $X$. Then
$x_{i}\setminus x_{j}= g^{-1}_{i}(j)$.
\end{lem}
\begin{proof}
If $S(i,a)=(j,b)$, then $x_{i}\setminus x_{j}=a$.
But by definition of $S$, we have that $S(i,a)=(g_{i}(a),f_{a}(i))$,
 so  $g_{i}(a)=j$
 which gives $a=g_{i}^{-1}(j)$.
\end{proof}
\begin{lem}\label{formule}
Assume $(X,S)$ is non-degenerate and involutive. Let $x_{i}, x_{k}$ be elements in $X$.  Then
$g^{-1}_{k}(i)=f_{g^{-1}_{i}(k)}(i)$
\end{lem}
\begin{proof}
Since $S$ is involutive, we have from claim \ref{debut_form} that for every $x_{i}, x_{j} \in X$, $g_{g_{i}(j)}f_{j}(i)=i$. We replace in this formula  $j$  by $g^{-1}_{i}(k)$ for some $1 \leq k \leq n$, then we obtain $i= g_{g_{i}(g^{-1}_{i}(k))}f_{g^{-1}_{i}(k)}(i)=g_{k}f_{g^{-1}_{i}(k)}(i)$.
So, we have $g^{-1}_{k}(i)=f_{g^{-1}_{i}(k)}(i)$.
\end{proof}
\begin{prop}\label{M_c1&c2}
Assume $(X,S)$ is non-degenerate, involutive and braided. Every triple  $(x_{i},x_{k},x_{m})$ of generators satisfies the following equation: $ (x_{i} \setminus x_{k} )\setminus (x_{i}\setminus x_{m})
=(x_{k} \setminus x_{i})\setminus(x_{k} \setminus
x_{m})$. Furthermore,  $M$ is coherent and satisfies the conditions $(C_{1})$ and $(C_{2})$.
\end{prop}
\begin{proof}
If $x_{i}=x_{k}$ or $x_{i}=x_{m}$ or $x_{k}=x_{m}$, then the equality holds trivially. So, assume that  $(x_{i},x_{k},x_{m})$ is a triple of different generators. This implies that $g_{i}^{-1} (k) \neq g_{i}^{-1} (m)$ and $g^{-1}_{k}(i) \neq g^{-1}_{k}(m)$, since the functions $g_{i}$ are bijective. Using the formulas for all different $1 \leq i,k,m \leq n$ from lemma \ref{form_compl}, we have:  $ (x_{i} \setminus x_{k} )\setminus (x_{i}\setminus x_{m})=
  g^{-1}_{x_{i} \setminus x_{k}} (x_{i}\setminus x_{m})=
   g^{-1}_{g_{i}^{-1}(k)} g_{i}^{-1} (m)$ and $ (x_{k} \setminus x_{i} )\setminus (x_{k}\setminus
x_{m})=
 g^{-1}_{x_{k} \setminus x_{i}} (x_{k}\setminus x_{m})=
 g^{-1}_{g_{k}^{-1}(i)} g_{k}^{-1} (m)$. We prove  that  $g^{-1}_{g_{i}^{-1}(k)} g_{i}^{-1} (m)=g^{-1}_{g_{k}^{-1}(i)} g_{k}^{-1} (m)$ by showing that
    $g_{i}g_{g_{i}^{-1}(k)}= g_{k}g_{g_{k}^{-1}(i)}$ for all  $1 \leq i,k \leq n$.
 Since $S$ is braided, we have from claim \ref{debut_form}, that
 $g_{i}g_{g_{i}^{-1}(k)}=g_{g_{i}(g_{i}^{-1}(k))}g_{f_{g_{i}^{-1}(k)}(i)}
 =$ $ g_{k}g_{f_{g_{i}^{-1}(k)}(i)}$.
 But, from lemma \ref{formule}, $f_{g^{-1}_{i}(k)}(i)=g^{-1}_{k}(i)$, so
 $g_{i}g_{g_{i}^{-1}(k)}= g_{k}g_{g^{-1}_{k}(i)}$.
 The monoid  $M$ is then coherent at $X$ but since $M$ is atomic, $M$ is
coherent. So, $M$  satisfies the conditions $(C_{1})$ and $(C_{2})$.
\end{proof}
Now, using the fact that $M$ satisfies $(C_{1})$ and $(C_{2})$, we show that it satisfies also $(C_{3})$.
\begin{prop}
Assume $(X,S)$ is non-degenerate, involutive and braided.
Then there is a finite generating set that is closed under
$\setminus$, that is  $M$ satisfies the condition $(C_{3})$.
\end{prop}
\begin{proof}
 From
claim \ref{cl_compl}, for any pair of generators $x_{i}, x_{j} $ there are unique $a,b \in X$ such that  $x_{i} a= x_{j} b$, that is   any pair of generators $x_{i}, x_{j} $ has
a right common multiple. Since from Proposition
\ref{M_c1&c2}, $M$ satisfies the condition $(C_{2})$,  we have
that $x_{i}$ and $x_{j}$ have a right lcm and the word $x_{i} a$ (or
$ x_{j} b$) represents the element $x_{i} \vee x_{j}$, since this is a common multiple of $x_{i}$ and
$x_{j}$ of least length. So,
 it holds that $x_{i}  \setminus x_{j}=a$ and $ x_{j}
\setminus x_{i} = b$, where $a,b \in X$. So,  $X \cup \{\epsilon\}$ is
closed under $\setminus$.
\end{proof}

\section{Additional properties of the structure group}
\subsection{The right lcm of the generators is a Garside element}
The braid groups and the Artin groups of finite type are Garside groups which satisfy the condition that the right lcm of their set of atoms is a Garside element. Dehornoy and Paris considered this additional condition as a part of the definition of Garside groups in \cite{deh_Paris} and in \cite{deh_francais} it was removed from the definition.
Indeed,  Dehornoy gives the example of a monoid that  is Garside and yet the right lcm of its atoms is not a Garside element \cite{deh_francais}. He shows  that the right lcm of the simple elements of a Garside monoid is a Garside element, where an element is  \emph{simple} if it belongs to the closure of a set of atoms under right complement and right lcm  \cite{deh_francais}.
We prove the following result:
\begin{thm}\label{thm_delta_lcm_atoms}
Let $G$ be the structure group of a non-degenerate, involutive and braided solution $(X,S)$ and let  $M$ be the monoid with the same presentation.
 Then the right lcm of the atoms (the elements of $X$) is a Garside element in $M$.
\end{thm}
In order to prove that, we show that the set of simple elements $\chi$, that is  the closure of $X$ under right complement and right lcm, is equal to the closure of $X$ under right  lcm (denoted by $\overline{X}^{\vee}$), where the empty word $\epsilon$ is added. So, this implies that $\Delta$, the right lcm of the simple elements, is the right lcm of  the elements in $X$.
We use the word reversing method developed by Dehornoy and the diagrams for word reversing. We illustrate in  example $1$ below the definition of the diagram and we refer the reader to \cite{deh_francais} and  \cite{deh_livre} for more details. Here,  reversing the word $u^{-1}v$ using the diagram amounts to computing a right lcm for the elements represented by $u$ and $v$ (see  \cite[p.65]{deh_livre}).
\begin{ex} Let us consider the monoid $M$ defined in example \ref{example_struct_gp}. We illustrate the construction of the reversing diagram.\\
 (a) The reversing diagram of  the word $x_{3}^{-1}x_{1}$  is constructed in figure \ref{des1}  in the following way. First, we begin with the left diagram and then using the  defining relation  $x_{1}x_{2}=x_{3}x_{4}$ in $M$, we complete it in the right diagram.  We have $x_{1}\setminus x_{3}=x_{2}$ and $x_{3}\setminus x_{1}=x_{4}$.

\begin{figure}[h]
  \includegraphics[scale=0.95]{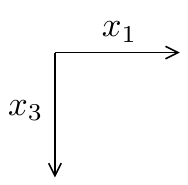}
  \includegraphics[scale=0.95]{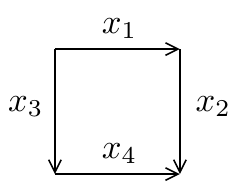}
  \caption{Reversing diagram of $x_{3}^{-1}x_{1}$}\label{des1}
\end{figure}

(b) The reversing diagram of  the word $x_{4}^{-2}x_{1}^{2}$ is described in figure \ref{des2}: we begin with the left diagram and then we complete it using the defining relations in the right diagram.
\begin{figure}[h]
 \includegraphics[scale=0.95]{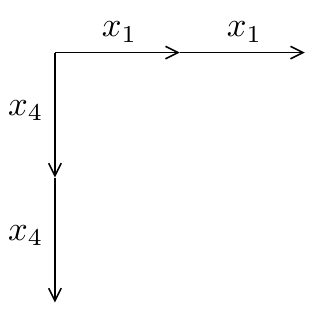}
  \includegraphics[scale=0.95]{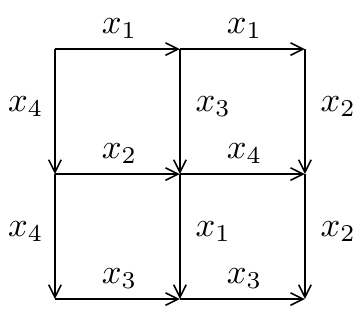}
\caption{Reversing diagram of $x_{4}^{-2}x_{1}^{2}$}\label{des2}
\end{figure}
     So, we have $x_{1}^{2}x_{2}^{2}=x_{4}^{2}x_{3}^{2}$ in $M$  and since  $x_{1}^{2}=x_{2}^{2}$ and $x_{3}^{2}=x_{4}^{2}$, it holds that  $x_{1}^{4}=x_{2}^{4}=x_{3}^{4}=x_{4}^{4}$ in $M$.
So, a word representing the right lcm of $x_{1}^{2}$ and $x_{4}^{2}$ is the word $x_{1}^{4}$ or the word $x_{4}^{4}$. We obtain from the diagram that the word $x_{2}^{2}$ represents the element $x_{1}^{2}\setminus x_{4}^{2}$ and the word $x_{3}^{2}$ represents the element  $x_{4}^{2}\setminus x_{1}^{2}$.
\end{ex}
In order to prove that  $\chi=\overline{X}^{\vee}\cup \{\epsilon\}$, we show that every complement of simple elements is the right lcm of some generators.
The following technical lemma is the basis of induction for the proof of Theorem \ref{thm_delta_lcm_atoms}.
\begin{lem}\label{lem_MXin X}
$(i)$ It holds that  $M\setminus X \subseteq X \cup \{\epsilon\}$.\\
$(ii)$ It holds that  $M\setminus (\vee_{j=1}^{j=k} x_{i_{j}}) \subseteq \overline{X}^{\vee}\cup \{\epsilon\}$, where $x_{i_{j}} \in X$, $ 1\leq j \leq k$.
\end{lem}
\begin{proof}
It holds that $S(X \times X) \subseteq X \times X$, so $X\setminus X \subseteq X\cup \{\epsilon\}$ and this implies inductively \emph{(i)} $M \setminus X \subseteq X \cup \{\epsilon\}$ (see the reversing diagram).\\
Let $u \in M$, then using the following rule of calculation on complements from \cite[Lemma 1.7]{deh_francais}: for every $u,v,w \in M$,  $u \setminus (v \vee w) = (u \setminus v) \vee (u \setminus w)$, we have inductively that $u \setminus (\vee_{j=1}^{j=k} x_{i_{j}}) = \vee_{j=1}^{j=k} (u \setminus x_{i_{j}})$.
From \emph{$(i)$}, $u \setminus x_{i_{j}}$ belongs to $X$, so
$\vee_{j=1}^{j=k} (u \setminus x_{i_{j}})$ is in $\overline{X}^{\vee}\cup \{\epsilon\}$. That is, $(ii)$ holds.
\end{proof}
Since the monoid $M$ is Garside, the set of simples  $\chi$  is finite and its  construction is  done in a finite number of steps in the following way:\\At the $0-$th step,  $\chi_{0}=X$.\\
At the first step,  $\chi_{1}=X \cup \{x_{i} \vee x_{j};$ $x_{i},x_{j} \in X \}\cup
\{x_{i} \setminus x_{j};$  $x_{i},x_{j} \in X \}$.\\
At the second step,  $\chi_{2}=\chi_{1} \cup \{u \vee v;$ $u,v \in \chi_{1} \} \cup \{u \setminus v;$ $u,v \in \chi_{1} \}$.\\
 We go on inductively and after a finite number of steps $k$, $\chi_{k}=\chi$.
\begin{prop}\label{prop_Simples}
It holds that $\chi=\overline{X}^{\vee}\cup \{\epsilon\}$.
\end{prop}
\begin{proof}
The proof is by induction on the number of steps $k$ in the construction of $\chi$. We show that each complement of simple elements is the right lcm of some generators. At the first step, we have that $\{x_{i} \setminus x_{j};$ for all  $x_{i},x_{j} \in X \}= X\cup \{\epsilon\}$. At the following steps, we do not consider the complements of the form $...\setminus x_{i}$ since these belong to $X$ (see lemma \ref{lem_MXin X}). At the second step, the complements have the following form
$x_{i} \setminus (x_{l} \vee x_{m})$ or $(x_{i} \vee x_{j}) \setminus (x_{l} \vee x_{m})$ and these belong to  $\overline{X}^{\vee}\cup \{\epsilon\}$   from lemma \ref{lem_MXin X}. Assume that at the $k-$th step, all the complements obtained belong to
$\overline{X}^{\vee}\cup \{\epsilon\}$, that is all the elements of $\chi_{k}$ are  right lcm of generators. At the $(k+1)-$th step, the complements have the following form $u \setminus v $, where $u,v \in \chi_{k}$. From the induction assumption,  $v$ is  a right lcm of generators, so  from lemma \ref{lem_MXin X},  $u \setminus v $  belongs to  $\overline{X}^{\vee}\cup \{\epsilon\}$.
\end{proof}

\begin{proof}[Proof of Theorem \ref{thm_delta_lcm_atoms}]
The right lcm of the set of simples $\chi$ is a Garside element and since from Proposition \ref{prop_Simples},  $\chi=\overline{X}^{\vee}\cup \{\epsilon\}$, we have that the right lcm of $X$ is a Garside element.
\end{proof}
We show now that the length of a Garside element $\Delta$ is equal to $n$, the cardinality of the set $X$. In order to show that, we prove in the following  that the right lcm of $k$ different generators has length $k$.
\begin{rem}\label{rem_interpret_compl_g}
When $w\setminus x$ is not equal to the empty word, then we can interpret $w\setminus x$ in terms of the functions $g^{-1}_{*}$ using the reversing diagram corresponding to  the words $w=h_{1}h_{2}..h_{k}$ and $x$, where $h_{i},x \in X$ and for brevity of notation we write $g^{-1}_{i}(x)$ for $g^{-1}_{h_{i}}(x)$:\\
\includegraphics{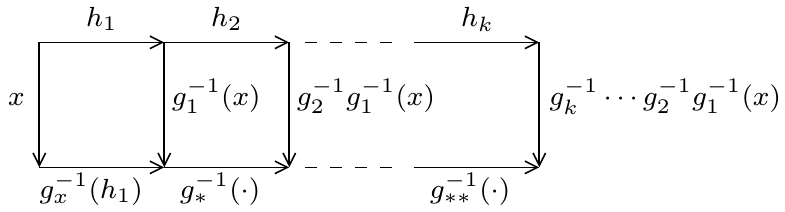}

That is, $h_{1}h_{2}..h_{k}\setminus x = g^{-1}_{k}..g^{-1}_{2}g^{-1}_{1}(x)$ and this is equal to $g^{-1}_{w}(x)$, since the action on $X$ is a right action.
Having a glance at the reversing diagram, we remark  that if  $w\setminus x$ is not equal to the empty word, then none of the complements $h_{1}\setminus x$, $h_{1}h_{2}\setminus x$,..,  $h_{1}h_{2}..h_{k-1}\setminus x$ can be equal to the empty word.
\end{rem}

\begin{lem}\label{lem_compl_egal}
Let $h_{i}, x$ be all different elements in $X$, for $1 \leq i \leq k$.
Then $(\vee_{i=1}^{i=k}h_{i}) \setminus x$ is not equal to the empty word $\epsilon$.
\end{lem}
\begin{proof}
By induction on $k$. If $k=1$, then   $h_{1}\setminus x \neq \epsilon$, as $h_{1} \neq x$.\\
Now assume $(\vee_{i=1}^{i=k-1}h_{i}) \setminus x \neq \epsilon$ and assume by contradiction that
$(\vee_{i=1}^{i=k}h_{i}) \setminus x= \epsilon$. Using the following rule of computation on the complement from \cite[Lemma 1.7]{deh_francais}: $(u \vee v)\setminus w=(u\setminus v)\setminus(u \setminus w)$, we have $(\vee_{i=1}^{i=k}h_{i}) \setminus x= (\vee_{i=1}^{i=k-1}h_{i} \vee h_{k}) \setminus x=
((\vee_{i=1}^{i=k-1}h_{i}) \setminus h_{k})\setminus ((\vee_{i=1}^{i=k-1}h_{i}) \setminus x)$.
From the induction assumption,  $(\vee_{i=1}^{i=k-1}h_{i}) \setminus x \neq \epsilon$ and $(\vee_{i=1}^{i=k-1}h_{i}) \setminus h_{k} \neq \epsilon$, so $(\vee_{i=1}^{i=k}h_{i}) \setminus x= \epsilon$ implies that
$(\vee_{i=1}^{i=k-1}h_{i}) \setminus h_{k} = (\vee_{i=1}^{i=k-1}h_{i}) \setminus x$.
Assume $w$ represents the element $\vee_{i=1}^{i=k-1}h_{i}$, then from remark \ref{rem_interpret_compl_g}, $w \setminus h_{k}$ can be interpreted as $g^{-1}_{w}(h_{k})$ and  $w \setminus x$  can be interpreted as $g^{-1}_{w}(x)$. The function  $g^{-1}_{w}$ is a bijective function as it is the composition of bijective functions, so $g^{-1}_{w}(h_{k})=g^{-1}_{w}(x)$ implies that $h_{k}=x$,  but this is a contradiction.
So, $(\vee_{i=1}^{i=k}h_{i}) \setminus x \neq  \epsilon$.
\end{proof}
\begin{thm}\label{thm_Garside_length_n}
Let $G$ be the structure group of a non-degenerate, braided and involutive solution $(X,S)$, where $X=\{x_{1},..,x_{n}\}$ and let  $M$ be the monoid with the same presentation.
Let $\Delta$ be a Garside element in $M$.
Then the length of $\Delta$ is  $n$.
\end{thm}
\begin{proof}
From theorem \ref{thm_delta_lcm_atoms}, $\Delta$ represents the right lcm of the elements in $X$, that is
$\Delta= x_{1} \vee x_{2} \vee ...\vee x_{n}$ in $M$.
We show by induction that a word representing the right lcm  $x_{i_{1}}\vee x_{i_{2}} \vee..\vee x_{i_{k}}$ has length  $k$, where $x_{i_{j}} \neq x_{i_{l}}$ for $j \neq l$. If $k=2$, then there are different generators $a,b$ such that $S(x_{i_{1}},a)=(x_{i_{2}},b)$, so $x_{i_{1}}a=x_{i_{2}}b$
is a relation in $M$ and the right lcm of $x_{i_{1}},x_{i_{2}}$ has length $2$.
Assume that the  right lcm  $x_{i_{1}}\vee x_{i_{2}} \vee..\vee x_{i_{k-1}}$ has length  $k-1$. Then the right lcm  $x_{i_{1}}\vee x_{i_{2}} \vee..\vee x_{i_{k-1}} \vee x_{i_{k}}$ is obtained from the reversing diagram corresponding to the  words $x_{i_{1}}\vee x_{i_{2}} \vee..\vee x_{i_{k-1}}$ and $x_{i_{k}}$.
From lemma \ref{lem_compl_egal}, $(x_{i_{1}}\vee x_{i_{2}} \vee..\vee x_{i_{k-1}})\setminus x_{i_{k}}$ is not equal to the empty word, so from lemma \ref{lem_MXin X} it has length $1$.
So, the right lcm  $x_{i_{1}}\vee x_{i_{2}} \vee..\vee x_{i_{k}}$ has length  $k$ and this implies that  $x_{1}\vee x_{2} \vee..\vee x_{n}$ has length  $n$.
\end{proof}

\subsection{Homological dimension}
 Dehornoy and Laffont construct a  resolution of $\Bbb Z$ (as trivial ${\Bbb Z}M$-module) by free
 ${\Bbb Z}M$-modules, when $M$ satisfies some conditions \cite{deh_homologie}. Moreover, they show that every Garside group is of type $FL$, that is with a finite resolution \cite[Prop. 2.9-2.10]{deh_homologie}.  Charney, Meier, and Whittlesey  show in \cite{charney} that Garside groups have finite homological dimension, using another approach. In \cite{deh_homologie}, Dehornoy and Laffont  show that whenever $M$ is a Garside monoid then the resolution defined in \cite{charney}  is isomorphic to the resolution they define.
We use the following result from \cite{deh_homologie} in
order to show that  the homological dimension of the structure
group corresponding to a set-theoretical solution $(X,S)$ of the
quantum Yang-Baxter equation is equal to the number of
generators in $X$.

\begin{prop}\cite[Cor.3.6]{deh_homologie}\label{limiter_dimension}
Assume that $M$ is a  locally Gaussian monoid admitting a
generating set $\chi$ such that $\chi \cup \{\epsilon\}$ is closed under left and right complement and
lcm and such that the norm of every element in $\chi$ is bounded
above by $n$. Then the (co)homological dimension of $M$ is at most $n$.
\end{prop}
Using Proposition \ref{limiter_dimension}, we prove the following result:
\begin{thm}
Let $(X,S)$ be  a set-theoretical solution  of the quantum
Yang-Baxter equation, where $X=\{x_{1},..,x_{n}\}$ and $(X,S)$ is non-degenerate, braided and involutive. Let $G$ be the structure group corresponding to $(X,S)$.
 Then  the (co)homological dimension of $G$ is equal to $n$, the number of generators in $X$.
\end{thm}
\begin{proof}
The set of simples $\chi$ satisfies the conditions of Corollary \ref{limiter_dimension} and the norm of every element in $\chi$ is bounded by $n$, since this is the length of the right lcm of $\chi$ (from Theorems \ref{thm_delta_lcm_atoms} and \ref{thm_Garside_length_n}).
So, the (co)homological dimension of $G$ is is equal to $n$.
\end{proof}
\begin{rem}
It was pointed to us by P.Etingof that this result can be also proved differently: by  showing that the classifying space of the structure group $G$ is a compact manifold of dimension $n$ (as there is a free action of $G$ on $R^{X}$).
\end{rem}
\section{Structure groups and indecomposable solutions}
Picantin defines the notion of   $\Delta$-pure Garside  monoid $M$ in \cite{picantin} and he shows there  that the center of $M$  is the infinite cyclic submonoid  generated by some power of $\Delta$.
We find in this section conditions under which a monoid is   $\Delta$-pure Garside in terms of set-theoretical solutions.
\subsection{$\Delta$-pure Garside monoids}
 Let $\chi$ be the set of simples and $\Delta$ a Garside element in $M$. The \emph{exponent} of $M$ is the order of the automorphism $\phi$, where $\phi$ is the extension of the function $x \rightarrow (x\setminus \Delta)\setminus \Delta$ from $\chi$ into itself.
 \begin{defn}\cite{picantin}
 The monoid $M$ is  $\Delta$\verb`-`\emph{pure} if for every $x,y$ in $X$, it holds that $\Delta_{x} = \Delta_{y}$, where $\Delta_{x} = \vee \{b \setminus x ; b \in M\}$ and  $\vee$ denotes the right lcm.
 \end{defn}
Picantin shows that if  $M$ is a   $\Delta$-pure Garside monoid with exponent $e$ and group of fractions $G$, then the center of $M$ (resp. of $G$) is the infinite cyclic submonoid (resp. subgroup) generated by $\Delta^{e}$.
Let consider the following example, to illustrate these definitions.
\begin{ex}\label{example_deltapure}
Let $X=\{x_{1},x_{2},x_{3}\}$ and let $S(x_{i},x_{j})=(f(j),f^{-1}(i))$, where $f=(1,2,3)$, be a non-degenerate, braided and involutive set-theoretical solution. Let $M$ be the monoid with the same presentation as the structure group of $(X,S)$, the defining relations in $M$ are then: $x_{1}^{2}=x_{2}x_{3}$, $x_{2}^{2}=x_{3}x_{1}$ and $x_{3}^{2}=x_{1}x_{2}$.
So, $X \setminus x_{1}=\{x_{3}\}$, $X \setminus x_{2}=\{x_{1}\}$ and $X \setminus x_{3}=\{x_{2}\}$.
Using the reversing diagram, we obtain inductively that  $M \setminus x_{i}=X\cup \{\epsilon\}$ for $1 \leq i \leq 3$, that is $M$ is   $\Delta$-pure Garside, since $\Delta_{1}=\Delta_{2}=\Delta_{3}$.
As an example, $x_{2}\setminus x_{1}=x_{3}$, so $x_{2}x_{1}\setminus x_{1}=x_{2}$ and so  $x_{2}x_{1}x_{3}\setminus x_{1}=x_{1}$ that is  $X \cup \{\epsilon\} \subseteq M \setminus x_{1}$ and since  $M \setminus x_{1}\subseteq X\cup \{\epsilon\}$ (see lemma \ref{lem_MXin X}) we have the equality.
Each word $x_{i}^{3}$ for $i=1,2,3$ represents a Garside element, denoted by $\Delta$. The set of simples is $\chi=\{\epsilon, x_{1},x_{2},x_{3},x_{1}^{2},x_{2}^{2},x_{3}^{2},\Delta\}$. The exponent of $M$  is equal to 1, since the function $x \rightarrow (x\setminus \Delta)\setminus \Delta$ from $\chi$ to itself is the identity.
As an example, the image of $x_{1}$ is $(x_{1}\setminus \Delta)\setminus \Delta=x_{1}^{2}\setminus \Delta=x_{1}$,
the image of $x_{2}^{2}$ is $(x_{2}^{2}\setminus \Delta)\setminus \Delta=x_{2}\setminus \Delta=x_{2}^{2}$ and so on.
So,  the center of the structure group of $(X,S)$ is cyclic and generated by $\Delta$, using the result of Picantin.
\end{ex}

\subsection{Structure groups and indecomposable solutions}
A non-degenerate, braided and involutive  set-theoretical solution  $(X,S)$ is said to be \emph{decomposable} if $X$ is a union of two nonempty disjoint non-degenerate invariant subsets, where   an \emph{invariant} subset $Y$ is a set that satisfies $S(Y \times Y)\subseteq Y \times Y$. Otherwise, $(X,S)$ is said to be \emph{indecomposable}.\\
Etingof et al give a classification of non-degenerate, braided and involutive solutions with $X$ up to $8$ elements, considering their decomposability and other properties \cite{etingof}.
 Rump proves Gateva-Ivanova's conjecture that every square-free, non-degenerate, involutive and braided solution $(X,S)$  is decomposable, whenever $X$ is finite. Moreover, he shows that an extension to infinite $X$ is false \cite{rump}.
We find a criterion for decomposability of the solution involving the Garside structure of the structure group (monoid), that is we prove the following result.
\begin{thm}\label{thm_deltapure_indecomp}
Let $G$ be the structure group of a non-degenerate, braided and involutive solution $(X,S)$ and let  $M$ be the monoid with the same presentation.
Then $M$ is   $\Delta$-pure Garside if and only if $(X,S)$ is indecomposable.
\end{thm}
In what follows, we use the notation from \cite{picantin}: for $X,Y \subseteq M$,  $Y \setminus X$ denotes the set of elements $b \setminus a$ for $a \in X$ and $b \in Y$. We write $Y \setminus a$ for $Y \setminus \{a\}$ and $b \setminus X$ for $\{b\} \setminus X$. We need the following lemma for the proof of Theorem \ref{thm_deltapure_indecomp}
\begin{lem}\label{lem_YYimpliqMY}
Let $(X,S)$ be the union of non-degenerate invariant subsets  $Y$ and $Z$. Then $M\setminus Y \subseteq Y \cup \{\epsilon\}$ and $M\setminus Z \subseteq Z\cup \{\epsilon\}$.
\end{lem}
\begin{proof}
If $Y$ is an invariant subset of $X$, then $Y\setminus Y \subseteq Y\cup \{\epsilon\}$, since $S(Y \times Y) \subseteq Y \times Y$.
From \cite[Proposition 2.15]{etingof}, the map $S$ defines bijections $Y \times Z \rightarrow Z \times Y$  and  $Z \times Y \rightarrow Y \times Z$. So, $S(Y \times Z) \subseteq Z \times Y$ and $S(Z \times Y) \subseteq Y \times Z$, and this implies $Z\setminus Y \subseteq Y$.
That is, we have that  $Y\setminus Y \subseteq Y\cup \{\epsilon\}$ and $Z\setminus Y \subseteq Y$, so $X\setminus Y \subseteq Y\cup \{\epsilon\}$ and this implies inductively that $M\setminus Y \subseteq Y\cup \{\epsilon\}$ (see the reversing diagram). The same holds for $Z$.
\end{proof}
\begin{proof}[Proof of Theorem \ref{thm_deltapure_indecomp}]
Assume $(X,S)$ is decomposable, that is $(X,S)$ is the union of non-degenerate invariant subsets  $Y$ and $Z$.
From lemma \ref{lem_YYimpliqMY}, we have $M\setminus Y \subseteq Y\cup \{\epsilon\}$ and $M\setminus Z \subseteq Z\cup \{\epsilon\}$.
 Let $y\in Y$ and $z\in Z$, then $\Delta_{y}= \vee (M\setminus y)$  cannot be the same as $\Delta_{z}= \vee (M\setminus z)$. So, $M$ is not   $\Delta$-pure Garside  .\\
Assume $M$ is not $\Delta$-pure Garside.  Let $x_{k} \in X$, we denote by $Y_{k}$ the set $(M\setminus x_{k})$ from which we remove $\{\epsilon\}$. So, $\Delta_{x_{k}}= \vee (Y_{k})$, where  $Y_{k}$ is a subset of $X$ from Lemma \ref{lem_MXin X}. Let $x_{i},x_{j}$ be in $X$, then  from \cite{picantin}, either $\Delta_{x_{i}}=\Delta_{x_{j}}$ or  the left gcd of $\Delta_{x_{i}}$ and $\Delta_{x_{j}}$ is $\epsilon$.  If $\Delta_{x_{i}}=\Delta_{x_{j}}$, then $Y_{i}=Y_{j}$ and if  the left gcd of $\Delta_{x_{i}}$ and $\Delta_{x_{j}}$ is $\epsilon$, then $Y_{i}$ and $Y_{j}$ are disjoint subsets of $X$. Since $M$ is not $\Delta$-pure Garside, there exist  $x_{i_{1}}$, $x_{i_{2}}$, .., $x_{i_{m}}$ in $X$ such that $Y_{i_{1}}$, $Y_{i_{2}}$,.., $Y_{i_{m}}$ are disjoint subsets of $X$. Moreover,
$X =Y_{i_{1}} \cup Y_{i_{2}} \cup..\cup Y_{i_{m}}$, since each $x \in X$ is equal to an element $x_{k} \setminus x_{i}$  for some $x_{k},x_{i} \in X$ (from the existence of left lcms). We show that $Y_{i_{j}}$ is an invariant subset of $X$, that is $S(Y_{i_{j}},Y_{i_{j}}) \subseteq (Y_{i_{j}},Y_{i_{j}})$. Let $x \in X$ and $y \in Y_{i_{j}}$, then $x \setminus y = x \setminus (w \setminus x_{i_{j}})$, for some $w \in M$. Using the following rule of computation on the complement from \cite[Lemma 1.7]{deh_francais}:  $x \setminus (u \setminus v) =(ux) \setminus v$, we have  $x \setminus y = (wx) \setminus  x_{i_{j}}$, that is  $x \setminus y$ belongs to $Y_{i_{j}}$. In particular, if $x \in Y_{i_{j}}$ then  $S(x,y')=(y,y'')$, where $y',y'' \in Y_{i_{j}}$. So, $Y_{i_{j}}$ is an invariant subset of $X$ for $1 \leq j \leq m$ and this implies that  $(X,S)$ is decomposable.
\end{proof}

\section{From  Garside groups to structure groups}
We establish the converse implication in the one-to-one correspondence between the Garside groups and the structure groups, that is we prove the following:
\begin{thm}\label{garside_struct_group}
Let $\operatorname{Mon} \langle X\mid R \rangle$ be a  Garside monoid such that:\\
$(i)$ There are $n(n-1)/2$ relations in $R$, where $n$ is the cardinality of $X$, and each side of a relation in $R$ has length 2.\\
$(ii)$ If the  word $x_{i}x_{j}$ appears in $R$, then it appears only once.\\
 Then there exists a function $S: X \times X \rightarrow X \times X$ such that $(X,S)$ is  a non-degenerate, involutive and braided set-theoretical solution and $\operatorname{Gp} \langle X\mid R \rangle$ is its structure group.\\
 If additionally: $(iii)$ There is no  word $x_{i}^{2}$  in $R$, then $(X,S)$ is square-free.
 \end{thm}
 \subsection{Proof of Theorem \ref{garside_struct_group}}
 In order to prove Theorem \ref{garside_struct_group}, we use the concepts of left lcm and left coherence
 from \cite{deh_francais} and \cite{deh_homologie}, but we do not use exactly the same notations.
   The notation for the \emph{left lcm of $x$  and $y$} is $z=x\widetilde{\vee} y$ and for \emph{the complement at left of $y$ on $x$} the notation is  $x \widetilde{\setminus} y$, where $x\widetilde{\vee} y=(x \widetilde{\setminus} y)y$.

\begin{defn}\cite{deh_homologie}
Let $M$ be a monoid. The \emph{left coherence condition} on $M$ is satisfied if it holds for any $x,y,z \in M$: $((x\widetilde{\setminus} y)\widetilde{\setminus} (z\widetilde{\setminus} y))\widetilde{\setminus} ((x\widetilde{\setminus} z)\widetilde{\setminus} (y\widetilde{\setminus} z))\equiv^{+} \epsilon$.
\end{defn}
This property is also called the \emph{left cube condition}.
We show that if $(X,S)$ is a non-degenerate and  involutive set-theoretical solution, then  $(X,S)$ is braided if and only if  $X$ is coherent and  left coherent.
 The left coherence of $X$  is satisfied if the following condition on all  $x_{i},x_{j},x_{k}$ in $X$ is satisfied: $(x_{i}\widetilde{\setminus} x_{j})\widetilde{\setminus} (x_{k}\widetilde{\setminus} x_{j})= (x_{i}\widetilde{\setminus} x_{k})\widetilde{\setminus} (x_{j}\widetilde{\setminus} x_{k})$, where the equality is in the free monoid since  the complement on the left is totally defined and  its range is $X$. Note that as  in the proof of the coherence,  left coherence on $X$ implies left coherence on $M$, since the monoid $M$ is atomic. Clearly, the following implication is derived from Theorem \ref{theo:garside}:
\begin{lem}\label{braided_implies_leftcoherent}
Assume  $(X,S)$ is non-degenerate and  involutive. If  $(X,S)$ is braided, then $X$ is coherent and left coherent.
\end{lem}

The proof of the converse implication is less trivial and requires a lot of computations. Before we proceed, we first express the left complement in terms of the functions $f_{i}$ that define $(X,S)$. As the proofs are symmetric to those done in Section 3.2 with the right complement we omit them.
\begin{lem}\label{compl_gauche}
 Assume $(X,S)$ is non-degenerate. Let $x_{i},x_{j}$ be different elements in $X$.
Then $x_{j}\widetilde{\setminus} x_{i}=f_{i}^{-1}(j)$.
\end{lem}

\begin{lem}\label{formule_gauche}
Assume $(X,S)$ is non-degenerate and  involutive. Let $x_{i}, x_{k}$ be elements in $X$.
Then $f^{-1}_{k}(i)=g_{f^{-1}_{i}(k)}(i)$.
\end{lem}

\begin{lem}\label{cor_Xcoherent_equations}
Assume $(X,S)$ is non-degenerate. If  $X$ is coherent and left coherent,  then for every $i,j,k$ the following equations hold:\\
(A) $f_{j}f_{f^{-1}_{j}(k)}=f_{k}f_{f_{k}^{-1}(j)}$\\
(B) $ g_{i}g_{g_{i}^{-1}(k)}=g_{k}g_{g_{k}^{-1}(i)}$
\end{lem}
\begin{proof}
From lemma \ref{compl_gauche}, we have for all different $1 \leq i,j,k \leq n$ that  $(x_{i}\widetilde{\setminus} x_{j})\allowbreak \widetilde{\setminus} (x_{k}\widetilde{\setminus} x_{j})=
  f^{-1}_{f^{-1}_{j}(k)} f_{j}^{-1} (i)$  and  $(x_{i}\widetilde{\setminus} x_{k})\widetilde{\setminus} (x_{j}\widetilde{\setminus} x_{k})=  f^{-1}_{f_{k}^{-1}(j)} f_{k}^{-1} (i)$.
  If $X$ is left coherent, then for all different $1 \leq i,j,k \leq n$, we have $(*)$  $f^{-1}_{f^{-1}_{j}(k)} f_{j}^{-1} (i)=$ $f^{-1}_{f_{k}^{-1}(j)} f_{k}^{-1} (i)$. If $j=k$, the  equality (A) holds trivially, so let fix $j$ and $k$ such that $j \neq k$. We denote
  $F_{1}=$ $f^{-1}_{f^{-1}_{j}(k)} f_{j}^{-1}$ and $F_{2}=$ $f^{-1}_{f_{k}^{-1}(j)} f_{k}^{-1}$, these functions are bijective, since these are compositions of bijective functions and satisfy  $F_{1}(i)=F_{2}(i)$ whenever $i \neq j,k$. It remains to show that  $F_{1}(k)=F_{2}(k)$ and $F_{1}(j)=F_{2}(j)$.  Assume by contradiction that $F_{1}(k)=F_{2}(j)$ and $F_{1}(j)=F_{2}(k)$, so there is $1 \leq m \leq n$ such that
  $m$= $f^{-1}_{f^{-1}_{j}(k)} f_{j}^{-1} (k)=$ $f^{-1}_{f_{k}^{-1}(j)} f_{k}^{-1} (j)$, that is $f_{f^{-1}_{j}(k)}(m)= f_{j}^{-1} (k)$ and $f_{f_{k}^{-1}(j)} (m)=f_{k}^{-1} (j)$.  That is $S(m,f^{-1}_{j}(k))=(m,f^{-1}_{j}(k))$ and $S(m,f^{-1}_{k}(j))=(m,f^{-1}_{k}(j))$,  since $(X,S)$ is involutive. So, $g_{m}\allowbreak (f^{-1}_{j}(k))=m$ and $g_{m}(f^{-1}_{k}(j))=m$. Since $g_{m}$ is bijective, this implies that there is $1 \leq l \leq n$ such that $l=$ $f^{-1}_{j}(k)=f^{-1}_{k}(j)$, that is $S(l,j)=(l,k)$.
 But, since $j\neq k$,  this contradicts the fact that $(X,S)$ is involutive.
  So, since the functions $f_{.}$ are bijective, (*) is equivalent to (A).
 Equation (B) is obtained in the same way using the coherence of $X$ (see lemma \ref{form_compl}).
 \end{proof}
\begin{prop}\label{coherence_implies_braided}
Let $(X,S)$ be  a non-degenerate and  involutive set-theoretical solution. If  $X$ is coherent and left coherent, then $(X,S)$ is braided.
\end{prop}
\begin{proof}
We need to show that the functions $f_{i}$ and $g_{i}$ satisfy the following equations from lemma \ref{debut_form}:\\
\emph{$(i)$}  $f_{j}f_{i}=f_{f_{j}(i)}f_{g_{i}(j)}$, $1 \leq i,j \leq n$. \\
\emph{$(ii)$} $g_{i}g_{j}=g_{g_{i}(j)}g_{f_{j}(i)}$, $1 \leq i,j \leq n$. \\
\emph{$(iii)$} $f_{g_{f_{l}(m)}(j)}g_{m}(l)=g_{f_{g_{l}(j)}(m)}f_{j}(l)$, $1 \leq j,l,m \leq n$.\\
From lemma \ref{cor_Xcoherent_equations}, we have for $1 \leq j, k \leq n$ that (A)  $f_{j}f_{f^{-1}_{j}(k)}=f_{k}f_{f^{-1}_{k}(j)}$. Assume $m=f^{-1}_{j}(k)$, that is $k=f_{j}(m)$ and we replace in formula (A) $f^{-1}_{j}(k)$ by $m$  and $k$ by $f_{j}(m)$, then we obtain $f_{j}f_{m}=f_{f_{j}(m)}f_{f^{-1}_{f_{j}(m)}(j)}$.
In order to show that \emph{$(i)$} holds, we  show that $f^{-1}_{f_{j}(m)}(j)=g_{m}(j)$. From lemma \ref{formule_gauche}, we have $f^{-1}_{l}(j)=g_{f^{-1}_{j}(l)}(j)$ for every $j,l$, so by replacing $l$ by  $f_{j}(m)$, we obtain $f^{-1}_{f_{j}(m)}(j)=g_{m}(j)$. So, \emph{$(i)$} holds.\\
From Corollary \ref{cor_Xcoherent_equations}, we have for $1 \leq j \neq k \leq n$ that (B) $ g_{i}g_{g_{i}^{-1}(k)}=g_{k}g_{g_{k}^{-1}(i)}$. Assume $m=g^{-1}_{i}(k)$, that is $k=g_{i}(m)$ and we replace in formula (B) $g^{-1}_{i}(k)$ by $m$  and $k$ by $g_{i}(m)$, then we obtain $g_{i}g_{m}=g_{g_{i}(m)}g_{g^{-1}_{g_{i}(m)}(i)}$.
In order to show that \emph{$(ii)$} holds, we show that $g^{-1}_{g_{i}(m)}(i)=f_{m}(i)$. From lemma \ref{formule}, we have $g^{-1}_{l}(i)=f_{g^{-1}_{i}(l)}(i)$, so by replacing $l$ by  $g_{i}(m)$, we obtain $g^{-1}_{g_{i}(m)}(i)=f_{m}(i)$. So, \emph{$(ii)$} holds.\\
It remains to show that \emph{$(iii)$} holds.
From \emph{$(i)$}, we have for $1 \leq i,j \leq n$ that  $f_{j}f_{i}=f_{f_{j}(i)}f_{g_{i}(j)}$ and this is equivalent to
$f_{g_{i}(j)}=f^{-1}_{f_{j}(i)}f_{j}f_{i}$.
We replace $i$ by $f_{l}(m)$ for some $1 \leq l,m \leq n$ in the formula.
We obtain $f_{g_{f_{l}(m)}(j)}=f^{-1}_{f_{j}f_{l}(m)}f_{j}f_{f_{l}(m)}$.
By applying these functions on  $g_{m}(l)$ on both sides, we obtain
$f_{g_{f_{l}(m)}(j)}g_{m}(l)=f^{-1}_{f_{j}f_{l}(m)}f_{j}f_{f_{l}(m)}g_{m}(l)$.
Since $(X,S)$ is involutive, we have $f_{f_{l}(m)}g_{m}(l)=l$ (see lemma \ref{debut_form}).  So,
$f_{g_{f_{l}(m)}(j)}g_{m}(l)=f^{-1}_{f_{j}f_{l}(m)}f_{j}(l)$.
From lemma \ref{formule_gauche}, we have $f^{-1}_{i}(k)=g_{f^{-1}_{k}(i)}(k)$ for every $i,k$, so  replacing $i$ by $f_{j}f_{l}(m)$ and $k$ by $f_{j}(l)$ gives
$f^{-1}_{f_{j}f_{l}(m)}(f_{j}(l))=g_{f^{-1}_{f_{j}(l)}f_{j}f_{l}(m)}(f_{j}(l))$.
So, $f_{g_{f_{l}(m)}(j)}g_{m}(l)=g_{f^{-1}_{f_{j}(l)}f_{j}f_{l}(m)}(f_{j}(l))$.
From \emph{$(i)$}, we have that $f^{-1}_{f_{j}(l)}f_{j}f_{l}(m)=f_{g_{l}(j)}(m)$,  so $f_{g_{f_{l}(m)}(j)}g_{m}(l)=g_{f_{g_{l}(j)}(m)}f_{j}(l)$, that is \emph{$(iii)$} holds.
\end{proof}

\begin{proof}[Proof of Theorem \ref{garside_struct_group}]

First,  we define a function $S: X \times X \rightarrow X \times X$ and $2n$ functions $f_{i},g_{i}$ for $1 \leq i \leq n$, such that $S(i,j)=(g_{i}(j),f_{j}(i))$ in the following way: if there is a relation $x_{i}x_{j}=x_{k}x_{l}$ then we define $S(i,j)=(k,l)$,  $S(k,l)=(i,j)$ and we define $g_{i}(j)=k$, $f_{j}(i)=l$, $g_{k}(l)=i$ and $f_{l}(k)=j$.
If the word $x_{i}x_{j}$ does not appear as a side of a relation, then we define $S(i,j)=(i,j)$ and we define $g_{i}(j)=i$ and $f_{j}(i)=j$.
We show that the functions $f_{i}$ and $g_{i}$ are well defined for $1 \leq i \leq n$:
assume $g_{i}(j)=k$ and $g_{i}(j)=k'$ for some $1 \leq j,k,k' \leq n$ and $k \neq k'$, then it means from the definition of $S$ that
$S(i,j)=(k,.)$ and $S(i,j)=(k',..)$ that is the word $x_{i}x_{j}$ appears twice in $R$ and this contradicts \emph{$(ii)$}. The same argument holds for the proof that the functions $f_{i}$  are well defined.\\
We show that the functions $f_{i}$ and $g_{i}$ are bijective for $1 \leq i \leq n$:
assume $g_{i}(j)=k$ and $g_{i}(j')=k$ for some $1 \leq j,j',k \leq n$ and $j \neq j'$, then  from the
definition of $S$ we have $S(i,j)=(k,l)$ and $S(i,j')=(k,l')$ for some $1 \leq l,l'\leq n$,  that is there are the following  two defining relations in $R$: $x_{i}x_{j}=x_{k}x_{l}$ and $x_{i}x_{j'}=x_{k}x_{l'}$.
But this means that $x_{i}$ and $x_{k}$ have two different right lcms and this contradicts the assumption that the monoid is Garside. So, these functions are injective and since $X$ is finite they are bijective.
Assuming  $f_{i}$ not injective yields  generators with two different left lcms.
So, $S$ is well-defined and $(X, S)$ is non-degenerate and from \emph{$(ii)$} $(X, S)$ is also involutive. It remains to show that $(X, S)$ is braided:
since $\operatorname{Mon} \langle X \mid R \rangle$ is Garside, it is coherent and left coherent so from lemma \ref{coherence_implies_braided}, $(X,S)$ is braided.
Obviously condition \emph{$(iii)$} implies that $(X,S)$ is also square-free.
\end{proof}
\subsection{The one-to-one correspondence}
It remains to establish the one-to-one correspondence  between  structure groups of  non-degenerate, involutive  and braided set-theoretical solutions of the quantum Yang-Baxter equation and a class of Garside groups admitting a certain presentation and in order to that  we need the following terminology and claims.
\begin{defn}
A \emph{tableau monoid} is a monoid $\operatorname{Mon} \langle X\mid R \rangle$  satisfying the condition that each side of a relation in $R$ has length 2.
 \end{defn}

 \begin{defn}
 We say that two tableau monoids $\operatorname{Mon} \langle X\mid R \rangle$ and $\operatorname{Mon} \langle X'\mid R' \rangle$ are \emph{t-isomorphic} if there exists a bijection $s:X \rightarrow X'$ such that $x_{i}x_{j}=x_{k}x_{l}$ is a relation in $R$ if and only if $s(x_{i})s(x_{j})=s(x_{k})s(x_{l})$ is a relation in $R'$.
 \end{defn}
 Clearly, if two tableau monoids are t-isomorphic then they are isomorphic and the definition is enlarged to groups.
Set-theoretical solutions  $(X,S)$ and  $(X',S')$  are \emph{isomorphic} if there exists a bijection $\phi: X \rightarrow X'$ which maps $S$ to $S'$, that is $S'(\phi(x),\phi(y))=(\phi(S_{1}(x,y)),\phi(S_{2}(x,y)))$.
\begin{prop}
Let $(X,S)$ and  $(X',S')$ be non-degenerate, involutive and braided set-theoretical solutions.
Assume $(X,S)$ and  $(X',S')$  are isomorphic.
Then their structure groups (monoids) $G$ and $G'$ are t-isomorphic tableau groups (monoids). Conversely, if  $\operatorname{Mon} \langle X\mid R \rangle$ and $\operatorname{Mon} \langle X\mid R' \rangle$ are t-isomorphic tableau Garside monoids each satisfying the conditions $(i)$ and $(ii)$ from Theorem \ref{garside_struct_group}, then the solutions  $(X,S)$ and  $(X',S')$  defined respectively by $\operatorname{Mon} \langle X\mid R \rangle$  and $\operatorname{Mon} \langle X\mid R' \rangle$   are isomorphic.
\end{prop}
\begin{proof}
Clearly, the structure groups (monoids) $G$ and $G'$ are tableau groups (monoids).
We need to show that $G$ and $G'$ are t-isomorphic.
Since $(X,S)$ and  $(X',S')$  are isomorphic, there exists a bijection $\phi: X \rightarrow X'$ which maps $S$ to $S'$, that is $S'(\phi(x),\phi(y))=(\phi(S_{1}(x,y)),\allowbreak \phi(S_{2}(x,y)))$.
So, since by definition $S(x,y)=(S_{1}(x,y),S_{2}(x,y))$, we have $xy=tz$ if and only if $\phi(x)\phi(y)=\phi(t)\phi(z)$. That is, if we take $s$ to be equal to $\phi$ we have that $G$ and $G'$ are t-isomorphic. For the converse, take $\phi$ to be equal to $s$ and from the definition of $S$ and $S'$ from their tableau we have
$S'(\phi(x),\phi(y))=(\phi(S_{1}(x,y)),\phi(S_{2}(x,y)))$, that is $(X,S)$ and  $(X',S')$  are isomorphic.
\end{proof}

\section{The structure group of a permutation solution}
In this part, we consider a special case of set-theoretical solutions of the quantum Yang-Baxter equation,
namely the permutation solutions. These solutions  were  defined by Lyubashenko (see \cite{etingof}). Let $X$ be a set and let $S:X^{2} \rightarrow X^{2}$ be a mapping. A permutation solution  is a set-theoretical solution of the form $S(x,y)=(g(y),f(x))$, where $f,g:X\rightarrow X$. The solution  $(X,S)$ is  nondegenerate iff $f,g$ are bijective,  $(X,S)$ is  braided  iff $fg=gf$ and  $(X,S)$ is involutive  iff $g=f^{-1}$. Note that these solutions are defined by only two functions, while for  a general set-theoretical solution the number of defining functions is twice the cardinality of the set $X$.
\subsection{About permutation solutions that are non-involutive}
In this subsection, we consider the special case of non-degenerate and braided permutation solutions that are not necessarily involutive and  we  show that their  structure group is Garside.

Let $X$ be a finite set and let $S:X^{2} \rightarrow X^{2}$ be defined by  $S(x,y)=(g(y),f(x))$, where $f,g:X\rightarrow X$ are bijective and satisfy  $fg=gf$. So, $(X,S)$ is a non-degenerate and braided permutation solution that is not necessarily  involutive, as we do not require  $g=f^{-1}$. Let $G$ be the structure group of  $(X,S)$ and let $M$ be the monoid with the same presentation.
We  define an equivalence relation on the set $X$ in the following way:\\
$x\equiv x'$ if and only if there is an integer $k$ such that $(fg)^{k}(x)=x'$. We define $X'=X/ \equiv$ and we define functions $f',g': X' \rightarrow X'$  such that  $f'([x])=[f(x)]$ and  $g'([x])=[g(x)]$, where $[x]$ denotes the equivalence class of $x$ modulo $\equiv$.
 We then define $S':X'\times X'\rightarrow X'\times X'$ by $S'([x],[y])= (g'([y]),f'([x])) = ([g(y)],[f(x)])$. Our aim is to show that $(X',S')$ is a well-defined non-degenerate, involutive and braided solution (a permutation solution) and that its structure group $G'$ is isomorphic to $G$. Before doing this,  we illustrate the main ideas of the proofs to come with an example.
 \begin{ex}
$X=\{x_{1},x_{2},x_{3},x_{4},x_{5}\}$ and let $f=(1,4)(2,3)$ and $g=(1,2)(3,4)$.
 Then $f,g$ are bijective and satisfy $fg=gf=(1,3)(2,4)$ but $fg \neq Id$, so $(X,S)$ is a non-degenerate and braided (permutation) solution, where  $S(x,y)=(g(y),f(x))$. The set of relations $R$ is:\\
$\begin{array}{ccc}
x_{1}^{2}=x_{2}x_{4}=x_{3}^{2}=x_{4}x_{2}&&
x_{1}x_{2}=x_{1}x_{4}=x_{3}x_{4}=x_{3}x_{2}\\
x_{2}^{2}=x_{1}x_{3}=x_{4}^{2}=x_{3}x_{1}&&
x_{1}x_{5}=x_{5}x_{4}=x_{3}x_{5}=x_{5}x_{2}\\
x_{2}x_{1}=x_{2}x_{3}=x_{4}x_{3}=x_{4}x_{1}&&
x_{2}x_{5}=x_{5}x_{3}=x_{4}x_{5}=x_{5}x_{1}\\
\end{array}$\\
Using $\equiv$ defined above, we have $X'=\{ [x_{1}],[x_{2}],[x_{5}]\}$, with $x_{1}\equiv x_{3}$
and $x_{2}\equiv x_{4}$, since in this example it holds that $fg(1)=3$ and $fg(2)=4$.
 Applying the definition of $S'$ yields  $S'([x_{1}],[x_{1}])=([g(1)],[f(1)])=([2],[4])=([2],[2])$ and so on. So, $G'= \operatorname{Gp}\langle[x_{1}],[x_{2}],[x_{5}] \mid
[x_{1}]^{2}=[x_{2}]^{2},[x_{1}][x_{5}]=[x_{5}][x_{2}],
[x_{2}][x_{5}]=[x_{5}][x_{1}]\rangle$.
Note that in $G$, it holds that $x_{1}=x_{3}$ and $x_{2}=x_{4}$ since many of the defining relations from $R$ involve cancellation and $G= \operatorname{Gp}\langle x_{1},x_{2},x_{5} \mid x_{1}^{2}=x_{2}^{2}, x_{1}x_{5}=x_{5}x_{2},x_{2}x_{5}=x_{5}x_{1} \rangle$. So, $G$ and $G'$ have the same presentation, up to a renaming of the generators.
\end{ex}
Before we proceed to the general case, we need the following general lemma. The proof, by induction on k, is omitted because it is straightforward and technical (see \cite{chou}).
\begin{lem}\label{lem:calcul_Sk}
If $k$ is even, then $S^{k}(x,y)=(f^{k/2}g^{k/2}(x),f^{k/2}g^{k/2}(y))$.\\
If $k$ is odd, then $S^{k}(x,y)=(f^{(k-1)/2}g^{(k+1)/2}(y),f^{(k+1)/2}g^{(k-1)/2}(x))$.
\end{lem}

\begin{lem}\label{equiv_cancel}
Let $x,x' \in X $. If $x \equiv x'$, then $x$ and $x'$ are equal in $G$.
\end{lem}
\begin{proof}
 Let  $x,x'$ be in  $X$ such that  $x \equiv x'$.
If $x \equiv x'$ then it means  that there is an integer $k$ such that
  $(fg)^{k}(x)=f^{k}g^{k}(x)=x'$.
 If $k$ is odd, then let $y$ in $X$ be defined in the following way:
  $y=f^{(k+1)/2}g^{(k-1)/2}(x)$.
  So, from lemma \ref{lem:calcul_Sk}, $S^{k}(x,y)=(f^{(k-1)/2}g^{(k+1)/2}(y),f^{(k+1)/2}g^{(k-1)/2}(x))
                  \allowbreak =(f^{(k-1)/2}g^{(k+1)/2}f^{(k+1)/2}g^{(k-1)/2}(x),y)=((fg)^{k}(x),y)=(x',y)$. So, there is a relation $xy=x'y$ in $G$ which implies that $x=x'$ in $G$.
  If $k$ is even and $(fg)^{k}(x)=x'$, then there is an element  $x'' \in X$ such that
   $(fg)^{k-1}(x)=x''$,  where $k-1$ is odd. So, from the subcase studied above,
    there is $y \in X$,  $y=f^{k/2}g^{(k-2)/2}(x)$,  such that there is a relation $xy=x''y$ in $G$ and this
     implies that $x=x''$ in $G$.
Additionally,  $(fg)(x'')=x'$, so from the same argument as above there is $z \in X$  such that there is  a relation $x'z=x''z$ in $G$
and this implies that $x'=x''$ in $G$. So, $x=x'$ in $G$.
\end{proof}
We now show that $(X',S')$ is a well-defined non-degenerate, involutive and braided solution and this  implies from Theorem \ref{theo:garside} that its structure group $G'$ is Garside.
\begin{lem}
$(i)$ $f'$ and $g'$ are well defined, so $S'$ is well-defined.\\
$(ii)$ $f'$ and $g'$ are bijective, so $S'$ is non-degenerate.\\
$(iii)$ $f'$ and $g'$ satisfy $f'g'=g'f'$, so $S'$ is braided.\\
(iv)  $f'$ and $g'$ satisfy $f'g'=g'f'=id_{X'}$, so $S'$ is involutive.
\end{lem}
\begin{proof}
\emph{$(i)$} Let $x,x'$ be in $X$ such that $x\equiv x'$.
We  show that $f'([x']) = f'([x])$, that is $[f(x')]=[f(x)]$.
Since  $x\equiv x'$, there is an integer $k$ such that $x'=(fg)^{k}(x)$, so $f'([x'])=[f(x')]=[f(fg)^{k}(x)]=[(fg)^{k}f(x)]=[f(x)]=f'([x])$, using the fact that $fg=gf$.
The same proof holds for $g'([x'])=g'([x])$.\\
\emph{$(ii)$} Assume that there are $x,y \in X$ such that $f'([x])=f'([y])$, that is $[f(x)]=[f(y)]$.
By the definition of $\equiv$, this means that there is an integer $k$ such that $f(x)=(fg)^{k}f(y)$, that is $f(x)=f(fg)^{k}(y)$,  since $fg=gf$. But $f$ is bijective, so $x=(fg)^{k}(y)$, which means that $[x]=[y]$. The same proof holds for $g'$, using the fact that $g$ is bijective.\\
\emph{$(iii)$}  Let $[x]\in X'$, so $f'(g'([x]))=f'([g(x)])=[f(g(x))]$  and $g'f'([x])=g'([f(x)])=[g(f(x))]$. Since $fg=gf$, $f'g'=g'f'$.\\
\emph{(iv)}  Let $[x]\in X'$, so from the definition of $\equiv$, we have $[fg(x)]=[x]$, so $f'g'([x])=[x]$, that is $f'g'=id_{X'}$.
\end{proof}
\begin{lem} \label{lem:G'isGarside}
Let $(X,S)$  be a not necessarily involutive permutation  solution. Let $X'=X/ \equiv$ and let $G'$ be the structure group corresponding to  $(X',S')$, as defined above.
Then $G'$ is  Garside. Furthermore,  if $x_{i}x_{j}=x_{k}x_{l}$  is  a defining relation in $G$, then $[x_{i}][x_{j}]=[x_{k}][x_{l}]$ is a defining relation in $G'$.
\end{lem}
\begin{proof}
It holds that $(X',S')$ is a non-degenerate, braided and involutive permutation  solution,
 so by Theorem \ref{theo:garside} the group $G'$ is Garside.  Assume that  $x_{i}x_{j}=x_{k}x_{l}$  is  a defining relation in $G$, that is $S(x_{i},x_{j})=(x_{k},x_{l})$.
 From the definition of $S'$, $S'([x_{i}],[x_{j}])=([g(x_{j})],[f(x_{i})])=([x_{k}],[x_{l}])$, that is there is a defining relation $[x_{i}][x_{j}]=[x_{k}][x_{l}]$ in $G'$. Note that this relation  may be a trivial one if
 $[x_{i}]=[x_{k}]$ and $ [x_{j}]=[x_{l}]$ in $G'$
\end{proof}
Now, it remains to show that the structure group $G$ is isomorphic
to  the group $G'$.
\begin{thm}
Let $G$ be the structure group of a non-degenerate and braided permutation  solution  $(X,S)$ that is not necessarily involutive. Then $G$ is a Garside group.
\end{thm}
\begin{proof}
We show that the group $G$ is isomorphic to the group  $G'$, where $G'$ is the structure group of $(X',S')$ and
$X'=X/ \equiv$, and from lemma \ref{lem:G'isGarside} this implies that $G$ is a Garside group.
Let  $\Phi: X \rightarrow X'$ be the quotient map defined by  $\Phi (x)=[x]$ for all $x \in X$.
From lemma \ref{lem:G'isGarside}, $\Phi:G \rightarrow G'$ is an homomorphism of groups, so $\Phi $ is an epimorphism.
We need to show that $\Phi$ is injective.
We show that  if $[x][y]=[t][z]$ is a non-trivial defining relation in $G'$, then $xy=tz$  is  a defining relation in $G$.
If $[x][y]=[t][z]$ is a non-trivial defining relation in $G'$, then since
 $S'([x],[y])= ([g(y)],[f(x)])$, we have that $[g(y)]=[t]$ and $[f(x)]=[z]$. That is, there are
$z' \in [z]$ and $t' \in [t]$ such that $g(y)=t'$ and $f(x)=z'$.
 This implies that $S(x,y)=(g(y),f(x))=(t',z')$, that is $xy=t'z'$  is  a defining relation in $G$.
 It holds that $t \equiv t'$ and $z \equiv z'$, so from lemma \ref{equiv_cancel}, $t=t'$ and $z=z'$ in $G$.
 So, $xy=tz$  is  a defining relation in $G$. \\
  Note that if $[x][y]=[t][z]$ is a trivial relation in $G'$, that is $[x]=[t]$ and $[y]=[z]$, then
 from lemma \ref{equiv_cancel} $x=t$ and $y=z$ in $G$ and so $xy=tz$ holds trivially in $G$.
So, $\Phi$ is an isomorphism of the groups $G$ and $G'$ and from lemma \ref{lem:G'isGarside},  $G$ is Garside.
\end{proof}

\subsection{Computation of $\Delta$ for a permutation solution}
In this subsection, we consider the structure group of a non-degenerate, braided and involutive permutation solution. We claim that given the decomposition of the defining function of the permutation solution as the product of disjoint cycles, one can easily find a Garside element in its structure group. This result can be extended to non-degenerate and braided permutation solutions that are not involutive, using the construction from section $7.1$.

 \begin{prop}
  Let $X=\{x_{1},..,x_{n}\}$,  and $(X,S)$ be a  non-degenerate, braided and involutive permutation solution defined by $S(i,j)=(f(j),f^{-1}(i))$, where  $f$ is a permutation on $\{1,..,n\}$. Let  $M$ be the monoid with the same presentation as the structure group. Assume that  $f$ can be described as the product of disjoint cycles:  \\ $f=(t_{1,1},..,t_{1,m_{1}})(t_{2,1},..,t_{2,m_{2}})
(t_{k,1},..,t_{k,m_{k}})(s_{1})..(s_{l})$, $t_{i,j},s_{*}\in \{1,..,n\}$. Then
$(i)$ For $1 \leq i \leq k$, $x_{t_{i,1}}^{m_{i}}=x_{t_{i,2}}^{m_{i}}=..=x_{t_{i,m_{i}}}^{m_{i}}$ in $M$ and this element is denoted by $x_{t_{i}}^{m_{i}}$.\\
$(ii)$ The element   $\Delta=x_{t_{1}}^{m_{1}}x_{t_{2}}^{m_{2}}..x_{t_{k}}^{m_{k}}x_{s_{1}}..x_{s_{l}}$
is a Garside element in $M$.
\end{prop}
We refer the reader to \cite{chou} for the proof.


\begin{thebibliography}{40}

\bibitem{charney} R. Charney, J. Meier and K. Whittlesey, Bestvina's normal form complex and the homology of Garside groups, Geometriae Dedicata \emph{105} (2004) 171-188.
\bibitem{chou} F. Chouraqui, Decision problems in Tableau-groups and Tableau-semigroups, PhD thesis, Technion, Israel.
        \bibitem{Clifford} A.H. Clifford, G.B. Preston, The algebraic theory of semigroups, vol.1, AMS Surveys \emph{7}, (1961).
\bibitem{deh_francais} P. Dehornoy, Groupes de Garside, Ann. Scient. Ec. Norm. Sup. \emph{35} (2002) 267-306.
\bibitem{deh_livre} P. Dehornoy, Braids and Self-Distributivity, Progress in Math. vol. 192, Birkhauser (2000).
\bibitem{deh_torsion}P. Dehornoy, Gaussian groups are torsion free,
 J. of Algebra \emph{210} (1998) 291-297.
 \bibitem{deh_complte}P. Dehornoy, Complete positive group presentations, J. of Algebra \emph{268} (2003) 156-197.
\bibitem{deh_homologie} P. Dehornoy, Y. Lafont,  Homology of Gaussian groups, Ann. Inst. Fourier (Grenoble) \emph{53-2} (2003) 489-540.
\bibitem{deh_Paris} P. Dehornoy, L. Paris, Gaussian groups and Garside groups, two generalizations of Artin groups, Proc. London Math. Soc. \emph{79-3} (1999) 569-604.
\bibitem{drinf} V.G. Drinfeld, On some unsolved problems in quantum group theory, Lec. Notes Math. \emph{1510} (1992) 1-8.
    \bibitem{etingof} P. Etingof, T. Schedler, A. Soloviev, Set-theoretical
solutions to the Quantum Yang-Baxter equation, Duke Math. J. \emph{100-2}
(1999) 169-209.
\bibitem{garside} F.A. Garside, The braid group and other groups,
Quart. J. Math. Oxford \emph{20} No.78 (1969) 235-254.
\bibitem{gateva} T. Gateva-Ivanova, A combinatorial approach to the set-theoric solutions of the Yang-Baxter equation, J. Math. Phys. \emph{45-10} (2005) 3828-3858.

\bibitem{gateva+van}T. Gateva-Ivanova, M. Van den Bergh, Semigroups of $I$-type, J. of Algebra \emph{206-1} (1998) 97-112.
\bibitem{jespers+okninski}E. Jespers, J. Okninski, Monoids and groups of $I$-type, Algebras and Representation Theory \emph{8} (2005) 709-729.
\bibitem{picantin} M. Picantin, The center of thin Gaussian groups,
J. of Algebra \emph{245} (2001) 92-122.
\bibitem{picantin_conjug} M. Picantin, The conjugacy problem in Garside groups, Comm. in Algebra \emph{29-3} (2001) 1021-1039.
\bibitem{picantin_torus} M. Picantin, Automatic structures for torus link groups, J. Knot Theory and its Ramifications \emph{12-6} (2003) 833-866.
\bibitem{rump} W. Rump, A decomposition theorem for square-free unitary solutions of the quantum Yang-Baxter equation, Advances in Math. \emph{193} (2005) 40-55.

  \end{thebibliography}
\end{document}